\documentclass{amsart}
\usepackage{amssymb}
\usepackage{amsfonts}

\newtheorem{theorem}{Theorem}
\theoremstyle{plain}

\newtheorem{corollary}{Corollary}

\newtheorem{definition}{Definition}
\newtheorem{example}{Example}

\newtheorem{lemma}{Lemma}

\newtheorem{proposition}{Proposition}
\newtheorem{remark}{Remark}

\numberwithin{equation}{section}

\begin{document}
\title[Continuity in Vector Metric Spaces]{Continuity in Vector Metric Spaces}
\author{C\"{u}neyt \c{C}ev\.ik}
\address{Department of Mathematics, Faculty of Science, Gazi
University, 06500 Teknikokullar Ankara, Turkey}
\email{ccevik@gazi.edu.tr}
\date{}

\begin{abstract}
We introduce vectorial and topological continuities for functions
defined on vector metric spaces and illustrate spaces of such
functions. Also, we describe some fundamental classes of vector
valued functions and extension theorems.
\end{abstract}

\maketitle

\section{Introduction and Preliminaries}

Let $E$ be a Riesz space. 
If every nonempty bounded above (countable) subset of $E$ has a supremum, then
$E$ is called \emph{Dedekind ($\sigma$-)complete}. A subset in $E$ is called \emph{order bounded} if it is bounded both above and below. We write $a_{n}\downarrow a$ if $(a_{n})$ is a decreasing sequence in $E$
such that $\inf a_{n}=a$. The Riesz space $E$ is said to be
\emph{Archimedean} if $n^{-1}a\downarrow 0$ holds for every $a\in E_{+}$. A
sequence $(b_{n})$ is said to be \emph{$o$-convergent} (or \emph{order convergent}) to $b$
if there is a sequence $(a_{n})$ in $E$ such that $a_{n}\downarrow 0$ and $%
\left\vert b_{n}-b\right\vert \leq a_{n}$ for all $n$, where $\left\vert
a\right\vert =a\vee (-a)$ for any $a\in E$. We will denote this order convergence by $b_{n}\overset{o}{\longrightarrow }b$. Furthermore, a sequence $(b_{n})$ is said to
be \emph{$o$-Cauchy} if there exists a sequence $(a_{n})$ in $E$
such that $a_{n}\downarrow 0$ and $\left\vert b_{n}-b_{n+p}\right\vert \leq
a_{n}$ for all $n$ and $p$. The Riesz space $E$ is said to be \emph{$o$-Cauchy complete} if
every $o$-Cauchy sequence is $o$-convergent. \vskip2mm

An operator $T:E\rightarrow F$ between two Riesz spaces is
\emph{positive} if $T(x)\geq 0$ for all $x\geq 0$. The operator
$T$ is said to be \emph{order bounded} if it maps order bounded
subsets of $E$ to order bounded subsets of $F$. The operator $T$
is called \emph{$\sigma$-order continuous} if
$x_{n}\overset{o}\longrightarrow 0$ in $E$ implies
$T(x_{n})\overset{o}\longrightarrow 0$ in $F$. Every
$\sigma$-order continuous operator is order bounded. If $T(x\vee
y)=T(x)\vee T(y)$ for all $x,y\in E$, the operator $T$ is a
\emph{lattice homomorphism}. For further information about Riesz
spaces and the operators on Riesz spaces, we refer to \cite{AB}
and \cite{Zaanen}. \vskip2mm

In \cite{CevikAltun}, a vector metric space is defined with a
distance map having values in a Riesz space. Also in
\cite{CevikAltun} and \cite{AltunCevik}, some results in metric
space theory are generalized to vector metric space theory, and
the Baire Theorem and some fixed point theorems in vector metric
spaces are given. Actually, the study of metric spaces having
value on a vector space has started by Zabrejko in
\cite{Zabrejko}. The distance map in the sense of Zabrejko takes
values from an ordered vector space. We use the structure of
lattice with the vector metrics having values in Riesz spaces,
then we have new results as mentioned above. This paper is on some
concepts and related results about continuity in vector metric
spaces. \vskip2mm

\begin{definition}
Let $X$ be a non-empty set and let $E$ be a Riesz space. The function $d:X\times
X\rightarrow E$ is said to be a vector metric (or $E$-metric) if it satisfies the following properties:

(vm1) $d(x,y)=0$ if and only if $x=y,$

(vm2) $d(x,y)\leq d(x,z)+d(y,z)$ for all $x,y,z\in X.$\\
Also the triple $(X,d,E)$ is said to be vector metric space.
\end{definition}

\begin{definition}
Let $(X,d,E)$ be a vector metric space.\vskip1mm
(a) A sequence $(x_{n})$ in $X$ is vectorial
convergent (or is $E$-convergent) to some $x\in X$, if there is a sequence $(a_{n})$ in $E$ such that $%
a_{n}\downarrow 0$ and $d(x_{n},x)\leq a_{n}$ for all $n$. We will denote this vectorial convergence by $x_{n}\overset{d,E}{\longrightarrow }x$\vskip1mm
(b) A sequence $(x_{n})$ in $X$ is called $E$-Cauchy sequence whenever there exists
a sequence $(a_{n})$ in $E$ such that $a_{n}\downarrow 0$ and $%
d(x_{n},x_{n+p})\leq a_{n}$ for all $n$ and $p$.\vskip1mm
(c) The vector metric space $X$ is called $E$-complete if each $E$-Cauchy
sequence in $X$ is $E$-convergent to a limit in $X$.
\end{definition}

One of the main goals of this paper is to demonstrate the properties of functions on vector metric spaces in a context more general than the continuity in metric analysis. Hence, the properties of Riesz spaces will be the tools for the study of continuity of vector valued functions. The result we get from here is continuity in general sense, an order property rather than a topological feature.\vskip2mm

In Section 2, we consider two types of continuity on vector metric spaces. This approach distinguishes continuities vectorially and topologically. Moreover, vectorial continuity examples are given and the relationship between vectorial continuity of a function and its graph is demonstrated. In Section 3, equivalent vector metrics, vectorial isometry, vectorial homeomorphism definitions and examples related to these concepts are given. In Section 4, uniform continuity is discussed and some extension theorems for functions defined on vector metric spaces are given. Finally, in Section 5, a uniform limit theorem on a vector metric space is given, and the structure of vectorial continuous function spaces is demonstrated.\vskip2mm

\section{Topological and Vectorial Continuity}

We now introduce the concept of continuity in vector metric spaces.

\begin{definition}
Let $(X,d,E)$ and $(Y,\rho ,F)$ be vector metric spaces, and let $x\in X$.\vskip2mm
(a) A function $f:X\rightarrow Y$ is said to be topological continuous at $x$ if for every $b>0$ in $F$ there exists some $a$ in $E$ such that $\rho (f(x),f(y))<b$ whenever $y\in X$
and $d(x,y)<a$. The function $f$ is said to be topological continuous if it is topological continuous at each point of $X$.\vskip2mm
(b) A function $f:X\rightarrow Y$ is said to be vectorial continuous at $x$ if $x_{n}\overset{d,E}{\longrightarrow }x$ in $X$ implies $f(x_{n})\overset{\rho ,F}{\longrightarrow }%
f(x)$ in $Y$. The function $f$ is said
to be vectorial continuous if it is vectorial continuous at each point of $X$.
\end{definition}

\vskip2mm

\begin{theorem}
\label{1}Let $(X,d,E)$ and $(Y,\rho ,F)$ be vector metric spaces
where $F$ is Archime-dean. If a function $f:X\rightarrow Y$ is
topological continuous, then $f$ is vectorial continuous.
\end{theorem}

\begin{proof}
Suppose that $x_{n}\overset{d,E}{\longrightarrow }x$. Then there
exists a sequence $(a_{n})$ in $E$ such that $a_{n}\downarrow 0$
and $d(x_{n},x)\leq a_{n}$ for all $n$. Let $b$ be any nonzero
positive element in $F$. By topological continuity of $f$ at $x$,
there exists some $a=a(b;y)$ in $E$ such that $y\in X$ and
$d(x,y)<a$ implies $\rho(f(x),f(y))<b$. Then
there exist elements $b_{n}=b_{n}(\frac{1}{n}b;x_{n})$ in $E$ such that $%
\rho (f(x_{n}),f(x))<\frac{1}{n}b$ whenever $d(x_{n},x)\leq a_{n}\wedge
b_{n}\leq a_{n}$ for all $n$. Since $F$ is Archimedean, $\frac{1}{n}%
b\downarrow 0$. Hence, $f(x_{n})\overset{\rho ,F}{\longrightarrow }f(x)$.
\end{proof}

\vskip2mm

Vectorial continuous functions have a number of nice characterizations.

\vskip2mm

\begin{corollary}
For a function $f:X\rightarrow Y$  between two vector metric spaces $(X,d,E)$
and $(Y,\rho ,F)$ the following statements hold:\vskip2mm

(a) If $F$ is Dedekind $\sigma $-complete and $f$ is vectorial continuous, then
$\rho (f(x_{n}),f(x))\downarrow 0$ whenever $d(x_{n},x)\downarrow 0$.\vskip2mm

(b) If $E$ is Dedekind $\sigma $-complete and $\rho
(f(x_{n}),f(x))\downarrow 0$ whenever $d(x_{n},x)\downarrow 0$, then the
function $f$ is vectorial continuous.\vskip2mm

(c) Suppose that $E$ and $F$ are Dedekind $\sigma $-complete. Then, the
function $f$ is vectorial continuous if and only if $\rho
(f(x_{n}),f(x))\downarrow 0$ whenever $d(x_{n},x)\downarrow 0$.
\end{corollary}

\begin{proof}
(a) If $d(x_{n},x)\downarrow 0$, then $x_{n}\overset{d,E}{\longrightarrow }x$. By the vectorial continuity of the function $f$, there is a sequence $(b_{n})$ in $F$ such that $b_{n}\downarrow 0$ and $\rho(f(x_{n}),f(x))\leq b_{n}$ for all $n$. Since $f$ is Dedekind $\sigma$-complete, $\rho (f(x_{n}),f(x))\downarrow 0$ holds.\vskip2mm

(b) Let $x_{n}\overset{d,E}{\longrightarrow }x$ in $X$. Then there is a sequence $(a_{n})$ in $E$ such that $a_{n}\downarrow 0$ and $d(x_{n},x)\leq a_{n}$ for all $n$. Since $E$ is Dedekind $\sigma$-complete, $d(x_{n},x)\downarrow 0$ holds. By the hypothesis, $\rho (f(x_{n}),f(x))\downarrow 0$, and so $f(x_{n})\overset{\rho,F}{\longrightarrow }f(x)$ in $Y$.\vskip2mm

(c) Proof is a consequence of (a) and (b).
\end{proof}

\vskip2mm

\begin{example}
Let $(X,d,E)$ be a vector metric space. If $x_{n}\overset{d,E}{%
\longrightarrow }x$ and $y_{n}\overset{d,E}{\longrightarrow }y$, then $d(x_{n},y_{n})%
\overset{o}{\longrightarrow }d(x,y)$, i.e, the vector metric map $d$ from $X^{2}$ to $E$ is
vectorial continuous. Here, $X^{2}$ is equipped with the $E$-valued vector metric $%
\tilde{d}$ defined as $\tilde{d}(z,w)=d(x_{1},x_{2})+d(y_{1},y_{2})$ for all $%
z=(x_{1},y_{1}),w=(x_{2},y_{2})\in X^{2}$, and $E$ is equipped with the absolute valued vector metric $\left\vert\cdot\right\vert$.
\end{example}

\vskip2mm

We recall that a subset $U$ of a vector metric space $(X,d,E)$ is called \emph{$E$-closed}
whenever $(x_{n})\subseteq U$ and $x_{n}\overset{d,E}{\longrightarrow }x$ imply $%
x\in U$.

\vskip2mm

\begin{theorem}
\label{theo3}
Let $(X,d,E)$ and $(Y,\rho ,F)$ be vector metric spaces. If a function $f:X\rightarrow Y$ is vectorial continuous, then for every $F$-closed subset $B$ of $Y$ the set $f^{-1}(B)$ is $E$-closed in $X$.
\end{theorem}

\begin{proof}
For any $x\in f^{-1}(B)$, there exists a sequence $(x_{n})$ in $f^{-1}(B)$ such that $x_{n}\overset{d,E}{\longrightarrow }x$. Since the function $f$ is $(E,F)$-continuous, $f(x_{n})\overset{\rho ,F}{\longrightarrow }f(x)$. But the set $B$ is $F$-closed, so $x\in f^{-1}(B)$. Then, the set $f^{-1}(B)$ is $E$-closed.
\end{proof}

\vskip2mm

If $E$ and $F$ are two Riesz spaces, then $E\times F$ is also a Riesz space with coordinatewise ordering defined as
\begin{equation*}
(e_{1},f_{1})\leq (e_{2},f_{2})\text{  }\Longleftrightarrow \text{  }e_{1}\leq e_{2}\ \text{ and
}\ f_{1}\leq f_{2}
\end{equation*}
for all $(e_{1},f_{1}),(e_{2},f_{2})\in E\times F$. The Riesz space $E\times F$ is a vector metric space, equipped with the \emph{biabsolute valued vector metric} $\left\vert\cdot\right\vert$ defined as
\begin{equation*}
\left\vert a-b\right\vert=(\left\vert e_{1}-e_{2}\right\vert,\left\vert f_{1}-f_{2}\right\vert)
\end{equation*}
for all $a=(e_{1},f_{1})$, $b=(e_{2},f_{2})\in E\times F$.

\vskip2mm

Let $d$ and $\rho$ be two vector metrics on $X$ which are $E$-valued and $F$-valued respectively. Then the map $\delta$ defined as
\begin{equation*}
\delta(x,y)=(d(x,y),\rho(x,y))
\end{equation*}
for all $x,y\in X$ is an $E\times F$-valued vector metric on $X$. We will call $\delta$ \emph{double vector metric}.

\vskip4mm

\begin{example}
Let $(X,d,E)$ and $(X,\rho ,F)$ be vector metric spaces.\vskip2mm
(i) Suppose that $f:X\rightarrow E$ and $g:X\rightarrow F$ are
vectorial continuous functions. Then the function $h$ from $X$ to
$E\times F$ defined by $h(x)=(f(x),g(x))$ for all $x\in X$ is
vectorial continuous with the double vector metric $\delta$ and
the biabsolute valued vector metric
$\left\vert\cdot\right\vert$.\vskip2mm

(ii) Let $p_{E}:E\times F\rightarrow E$ and $p_{F}:E\times F\rightarrow F$ be the projection maps. Any function $h:X\rightarrow E\times F$ can be written as $h(x)=(f(x),g(x))$ for all $x\in X$ where $f=p_{E}\circ h$ and $g=p_{F}\circ h$. If $h$ is vectorial continuous, so are $f$ and $g$ since $p_{E}$ and $p_{F}$ are vectorial continuous.
\end{example}

\vskip2mm

Let $(X,d,E)$ and $(Y,\rho ,F)$ be vector metric spaces. Then $X\times Y$ is a vector metric space, equipped with the $E\times F$-valued \emph{product vector metric} $\pi$ defined as
\begin{equation*}
\pi(z,w)=(d(x_{1},x_{2}),\rho(y_{1},y_{2}))
\end{equation*}
for all $z=(x_{1},y_{1})$, $w=(x_{2},y_{2})\in X\times Y$.\vskip2mm

Consider the vector metric space $(X\times Y,\pi,E\times F)$. The projection maps $p_{X}$ and $p_{Y}$ defined on $X\times Y$ which are $X$-valued and $Y$-valued respectively are vectorial continuous. For any function $f$ from a vector metric space $(Z,\delta,G)$ to $(X\times Y,\pi,E\times F)$, the function $f$ is vectorial continuous if and only if both $p_{X}\circ f$ and $p_{Y}\circ f$ are vectorial continuous.

\vskip2mm

\begin{example}
Let $(X,d,E)$ and $(Y,\rho ,E)$ be vector metric spaces. Suppose
that $f:X\rightarrow E$ and $g:Y\rightarrow E$ are vectorial
continuous functions. Then the function $h$ from $X\times Y$ to
$E$ defined by $h(x,y)=\left\vert f(x)-g(y)\right\vert$ for all
$x\in X$, $y\in Y$ is vectorial continuous with the product vector
metric $\pi$ and the absolute valued vector metric
$\left\vert\cdot\right\vert$.
\end{example}

\vskip2mm

\begin{example}
Let $(X,d,E)$ and $(Y,\rho ,F)$ be vector metric spaces. Suppose
that $f:X\rightarrow E$ and $g:Y\rightarrow F$ are vectorial
continuous functions. Then the function $h$ from $X\times Y$ to
$E\times F$ defined by $h(x,y)=(f(x),g(y))$ for all $x\in X$,
$y\in Y$ is vectorial continuous with the product vector metric
$\pi$ and the biabsolute valued vector metric
$\left\vert\cdot\right\vert$.
\end{example}

\vskip2mm

The last three examples inspire the following results.

\vskip4mm

\begin{corollary}
(a) If $f:(X,d,E)\rightarrow (Y,\eta,G)$ and
$g:(X,\rho,F)\rightarrow (Z,\xi,H)$ are vectorial continuous
functions, then the function $h$ from $X$ to $Y\times Z$ defined
by $h(x)=(f(x),g(x))$ for all $x\in X$ is vectorial continuous
with the $E\times F$-valued double vector metric $\delta$ and the
$G\times H$-valued product vector metric $\pi$.\vskip2mm

(b) Let $G$ be a Riesz space. If $f:(X,d,E)\rightarrow G$ and
$g:(Y,\rho,F)\rightarrow G$ are vectorial continuous functions,
then the function $h$ from $X\times Y$ to $G$ defined by
$h(x,y)=\left\vert f(x)-g(y)\right\vert$ for all $x\in X$, $y\in
Y$ is vectorial continuous with the $E\times F$-valued product
vector metric $\pi$ and the absolute valued vector metric
$\left\vert\cdot\right\vert$.\vskip2mm

(c) If $f:(X,d,E)\rightarrow (Z,\eta,G)$ and
$g:(Y,\rho,F)\rightarrow (W,\xi,H)$ are vectorial continuous
functions, then the function $h$ from $X\times Y$ to $Z\times W$
defined by $h(x,y)=(f(x),g(y))$ for all $x\in X$, $y\in Y$ is
vectorial continuous with the $E\times F$-valued and $G\times
H$-valued product vector metrics.
\end{corollary}

\vskip2mm

We have the next proposition for any product vector metric.

\vskip2mm

\begin{proposition}
\label{prop}Let $(z_{n})=(x_{n},y_{n})$ be a sequence in $(X\times Y,\pi,E\times F)$ and let $z=(x,y)\in X\times Y$. Then, \ $z_{n}\overset{\pi,E\times F}{\longrightarrow }z$ \ if and only if \ $x_{n}\overset{d,E}{\longrightarrow }x$ \ and \ $y_{n}\overset{\rho,F}{\longrightarrow }y.$
\end{proposition}

\vskip2mm

Now let us give a relevance between vectorial continuity of a function and being closed of its graph.

\vskip2mm

\begin{corollary}
Let $(X,d,E)$ and $(Y,\rho ,F)$ be vector metric spaces and let $f$ be a function from $X$ to $Y$. For the graph $G_{f}$ of $f$, the following statements hold.\vskip2mm

(a) The graph $G_{f}$ is $E\times F$-closed in $(X\times Y,\pi,E\times F)$ if and only if for every sequence $(x_{n})$ with $x_{n}\overset{d,E}{\longrightarrow }x$ and $f(x_{n})\overset{\rho,F}{\longrightarrow }y$ we have $y=f(x)$.\vskip2mm

(b) If the function $f$ is vectorial continuous, then the graph $G_{f}$ is $E\times F$-closed.\vskip2mm

(c) If the function $f$ is vectorial continuous at $x_{0}\in X$, then the induced function $h:X\rightarrow G_{f}$ defined by $h(x)=(x,f(x))$ is vectorial continuous at $x_{0}\in X$.
\end{corollary}

\begin{proof}
For the proof of (a), suppose that the graph $G_{f}$ is $E\times F$-closed. If $x_{n}\overset{d,E}{\longrightarrow }x$ and $f(x_{n})\overset{\rho,F}{\longrightarrow }y$, then we have $(x_{n},f(x_{n}))\overset{\pi,E\times F}{\longrightarrow }(x,y)$ by Proposition \ref{prop}. Hence $(x,y)\in G_{f}$, and so $y=f(x)$. Conversely, suppose that $(z_{n})=(x_{n},f(x_{n}))$ is a sequence in $G_{f}$ such that $z_{n}\overset{\pi,E\times F}{\longrightarrow }z=(x,y)\in X\times Y$. By Proposition \ref{prop}, $x_{n}\overset{d,E}{\longrightarrow }x$ and $f(x_{n})\overset{\rho,F}{\longrightarrow }y$. Then $y=f(x)$, and so $z\in G_{f}$.\vskip2mm
Proofs of (b) and (c) are similar to the proof of (a).
\end{proof}

\vskip2mm

\section{Fundamental Vector Valued Function Classes}

\begin{definition}
The $E$-valued vector metric $d$ and $F$-valued vector metric $\rho$ on $X$ are said to be $(E,F)$-equivalent if for any $x\in X$ and any sequence $(x_{n})$ in $X$, $$x_{n}\overset{d,E}{\longrightarrow }x \text{\ \ \ if and only if\ \ \ } x_{n}\overset{\rho,F}{\longrightarrow }x.$$
\end{definition}

\vskip2mm

\begin{lemma}
For any two $E$-valued vector metrics $d$ and $\rho$ on $X$, the
following statements are equivalent:

(a) There exist some $\alpha,\beta>0$ in $\mathbb{R}$ such that
$\alpha d(x,y)\leq\rho (x,y)\leq\beta d(x,y)$ for all $x,y\in X$.

(b) There exist two positive and $\sigma$-order continuous
operators $T$ and $S$ from $E$ to itself such that $\rho (x,y)\leq
T(d(x,y))$ and $d(x,y)\leq S(\rho (x,y))$ for all $x,y\in X$.

\end{lemma}

\begin{proof}
Let $T$ and $S$ be two operators defined as $T(a)=\beta a$ and
$S(a)=\alpha^{-1}a$ for all $a\in E$. If (a) holds, then $T$ and
$S$ are positive and $\sigma$-order continuous operators, and
satisfy $\rho (x,y)\leq T(d(x,y))$ and $d(x,y)\leq S(\rho (x,y))$
for all $x,y\in X$. Conversely, (a) holds since every
($\sigma$-)order continuous operator is order bounded (\cite{AB},
1.54).
\end{proof}

\vskip2mm

Now, we give the following result for the equivalence of vector metrics.

\vskip2mm

\begin{theorem}
An $E$-valued vector metric $d$ and a $F$-valued vector metric $\rho$ on $X$ are $(E,F)$-equivalent if there exist positive and $\sigma$-order continuous two operators $T:E\rightarrow F$, $S:F\rightarrow E$ such that
\begin{equation}
\rho (x,y)\leq T(d(x,y))\text{\ \ \ and\ \ \ }d(x,y)\leq S(\rho (x,y))\label{2}
\end{equation}
 for all $x,y\in X$.
\end{theorem}

\vskip2mm

\begin{example}
Suppose that the ordering of \ $\mathbb{R}^{2}$ is coordinatewise.\vskip2mm
(a) Let $d$ and $\rho$ be $\mathbb{R}$-valued and $\mathbb{R}^{2}$-valued vector metrics on $\mathbb{R}$ respectively, defined as
\begin{equation*}
d(x,y)=a\left\vert x-y\right\vert\ \ \ and\ \ \ \rho(x,y)=(b\left\vert x-y\right\vert,c\left\vert x-y\right\vert)
\end{equation*}
where $b,c\geq 0$ and $a,b+c>0$. Consider the two operators \ $T:\mathbb{R}\rightarrow \mathbb{R}^{2}\ ;\ T(x)=a^{-1}(b\:x,c\:x)$ \ and \ $S:\mathbb{R}^{2}\rightarrow \mathbb{R}\ ;\ S(x,y)=a\:b^{-1}x$ for all $x,y\in\mathbb{R}$. Then, the operators $T$ and $S$ are positive and $\sigma$-order continuous, and (\ref{2}) is satisfied. Hence, the metrics $d$ and $\rho$ are $(\mathbb{R},\mathbb{R}^{2})$-equivalent on $\mathbb{R}$.\vskip2mm
(b) Let $d$ and $\rho$ be $\mathbb{R}$-valued and $\mathbb{R}^{2}$-valued vector metrics on $\mathbb{R}^{2}$ respectively, defined as
\begin{equation*}
d(x,y)=a\left\vert x_{1}-y_{1}\right\vert+b\left\vert x_{2}-y_{2}\right\vert\ \ \ and\ \ \ \rho(x,y)=(c\left\vert x_{1}-y_{1}\right\vert,e\left\vert x_{2}-y_{2}\right\vert)
\end{equation*}
where $x=(x_{1},x_{2})$, $y=(y_{1},y_{2})$ and $a,b,c,e>0$. Let $T:\mathbb{R}\rightarrow \mathbb{R}^{2}$ and $S:\mathbb{R}^{2}\rightarrow \mathbb{R}$ be two operators defined as \ $T(x)=(c\:a^{-1}\:x,e\:b^{-1}x)$ \ and \ $S(x,y)=a\:c^{-1}x+b\:e^{-1}y$. Then, the operators $T$ and $S$ are positive and $\sigma$-order continuous. The condition (\ref{2}) is satisfied. So, the vector metrics $d$ and $\rho$ are $(\mathbb{R},\mathbb{R}^{2})$-equivalent on $\mathbb{R}^{2}$. On the other hand, if $\eta$ is another $\mathbb{R}$-valued vector metric on $\mathbb{R}^{2}$ defined as
\begin{equation*}
\eta(x,y)=\max\left\{a\left\vert x_{1}-y_{1}\right\vert,b\left\vert x_{2}-y_{2}\right\vert\right\}
\end{equation*}
where $x=(x_{1},x_{2})$, $y=(y_{1},y_{2})$ and $a,b>0$, and the operator $S$ is defined as $S(x,y)=\max\left\{a\:c^{-1}x,b\:e^{-1}y\right\}$, then the vector metrics $\eta$ and $\rho$ are $(\mathbb{R},\mathbb{R}^{2})$-equivalent on $\mathbb{R}^{2}$.
\end{example}

\vskip2mm

\begin{remark}
Vectorial continuity is invariant under equivalent vector metrics.
\end{remark}

\vskip2mm

Let us show how an isometry is defined between two vector metric spaces.

\vskip2mm

\begin{definition}
Let $(X,d,E)$ and $(Y,\rho ,F)$ be vector metric spaces. A
function $f:X\rightarrow Y$ is said to be a vector isometry if
there exists a linear operator $T_{f}:E\rightarrow F$ satisfying
the following two conditions, \vskip0mm (i) $T_{f}(d(x,y))=\rho
(f(x),f(y))$ for all $x,y\in X$,\vskip0mm
(ii) $T_{f}(a)=0$ implies $a=0$ for all $a\in E$.\\
If the function $f$ is onto, and the operator $T_{f}$ is a lattice homomorphism, then the vector metric spaces $(X,d,E)$ and $(Y,\rho ,T_{f}(E))$ are called vector isometric.
\end{definition}

\vskip2mm

\begin{remark}
A vector isometry is a vectorially distance preserving one-to-one function.
\end{remark}

\vskip2mm

\begin{example}
Let $d$ be $\mathbb{R}$-valued vector metric and let $\rho$ be $\mathbb{R}^{2}$-valued vector metric on $\mathbb{R}$ defined as
\begin{equation*}
d(x,y)=a\left\vert x-y\right\vert\ \ \ and\ \ \ \rho(x,y)=(b\left\vert x-y\right\vert,c\left\vert x-y\right\vert)
\end{equation*}
where $b,c\geq 0$ and $a,b+c>0$. Consider the identity mapping $I$ on $\mathbb{R}$ and the operator $T_{I}:\mathbb{R}\rightarrow \mathbb{R}^{2}$ defined as $T_{I}(x)=a^{-1}(bx,cx)$ for all $x\in\mathbb{R}$. Then, the identity mapping $I$ is a vector isometry. So, the vector metric spaces $(\mathbb{R},d,\mathbb{R})$ and $(\mathbb{R},\rho,\left\{(x,y):cx=by;\:x,y\in\mathbb{R}\right\})$ are vector isometric.
\end{example}

\vskip2mm

\begin{definition}
Let $(X,d,E)$ and $(Y,\rho ,F)$ be vector metric spaces. A function $f:X\rightarrow Y$ is said to be a vector homeomorphism if $f$ is a one-to-one and vectorial continuous and has a vectorial continuous inverse on $f(X)$. If the function $f$ is onto, then the vector metric spaces $X$ and $Y$ are called vector homeomorphic.
\end{definition}

\vskip2mm

\begin{remark}
A vector homeomorphism is one-to-one function that preserves vectorial convergence of sequences.
\end{remark}

\vskip2mm

By Theorem \ref{theo3}, we can develop another characterization result for vector homeomorphisms.

\vskip2mm

\begin{theorem}
An onto vector homeomorphism is a one-to-one function that preserves vector closed sets.
\end{theorem}

\begin{proof}
Let $f:X\rightarrow Y$ be an $(E,F)$-homeomorphism. Since $f$ is a one-to-one function and its inverse $f^{-1}$ is vectorial continuous, by Theorem \ref{theo3} for every $E$-closed set $A$ in $X$, $f(A)=(f^{-1})^{-1}(A)$ is $F$-closed in $Y$.
\end{proof}

\vskip2mm

The following example illustrates a relationship between vectorial equivalence and vector homeomorphism.

\begin{example}
Let $d$ and $\rho$ be two $(E,F)$-equivalent vector metrics on $X$. Then $(X,d,E)$ and $(X,\rho ,F)$ are vector homeomorphic under the identity mapping. On the other hand, if vector metric spaces $(X,d,E)$ and $(Y,\rho ,F)$ are vector homeomorphic under a function $f$, then the vector metrics $d$ and $\delta$ defined as $$\delta (x,y)=\rho (f(x),f(y))$$ for all $x,y\in X$ are $(E,F)$-equivalent vector metrics on $X$.
\end{example}

\section{Extension Theorems on Continuity}

If $X$ and $Y$ are vector metric spaces, $A\subseteq X$ and $f:A\rightarrow
Y $ is vectorial continuous, then we might ask whether there exists a vectorial
continuus extension $g$ of $f$. Below, we deal with some simple extension
techniques.

\vskip2mm

\begin{theorem}
\label{5}Let $(X,d,E)$ and $(Y,\rho ,F)$ be vector metric spaces, and let $f$ and $g$
be vectorial continuous functions from $X$ to $Y$. Then the set $\{x\in
X:f(x)=g(x)\}$ is an $E$-closed subset of $X$.
\end{theorem}

\begin{proof}
Let $B=\left\{x\in X:f(x)=g(x)\right\}$. Suppose $(x_{n})\subseteq B$ and $x_{n}\overset{d,E}{\longrightarrow }x$. Since $f$ and $g$ are vectorial continuous, there exist sequences $(a_{n})$ and $(b_{n})$ such that $a_{n}\downarrow 0$, $b_{n}\downarrow 0$ and $\rho (f(x_{n}),f(x))\leq a_{n}$, $\rho (g(x_{n}),g(x))\leq b_{n}$ for all $n$. Then $$\rho (f(x),g(x))\leq \rho (f(x_{n}),f(x))+\rho (f(x_{n}),g(x_{n}))+\rho (g(x_{n}),g(x))\leq a_{n}+b_{n}$$ for all $n$. So $f(x)=g(x)$, i.e $x\in B$. Hence, the set $B$ is an $E$-closed subset of $X$.
\end{proof}

\vskip2mm

The following corollary is a consequence of Theorem \ref{5}.

\vskip2mm

\begin{corollary}
Let $(X,d,E)$ and $(Y,\rho ,F)$ be vector metric spaces and let $f$ and $g$
be vectorial continuous functions from $X$ to $Y$. If the set $\{x\in
X:f(x)=g(x)\}$ is $E$-dense in $X$, then $f=g$.
\end{corollary}

\vskip2mm

\begin{definition}
Let $(X,d,E)$ and $(Y,\rho ,F)$ be vector metric spaces.\vskip2mm

(a) A function $f:X\rightarrow Y$ is said to be topological uniformly continuous on $X$ if for every $b>0$ in $F$
there exists some $a$ in $E$ such that for all $x,y\in X$, $\rho (f(x),f(y))<b$ whenever $d(x,y)<a$.\vskip2mm

(b) A function $f:X\rightarrow Y$ is said to be vectorial uniformly continuous on $X$ if for every $E$-Cauchy sequence $(x_{n})$ the sequence $f(x_{n})$ is $F$-Cauchy. \end{definition}

\vskip2mm

\begin{theorem}
\label{Theo5}Let $(X,d,E)$ and $(Y,\rho ,F)$ be vector metric spaces where $F$ is Archime-dean. If a function $f:X\rightarrow Y$ is topological uniformly continuous, then $f$ is vectorial uniformly continuous.
\end{theorem}

\begin{proof}
Suppose that $(x_{n})$ is an $E$-Cauchy sequence. Then there exists a
sequence $(a_{n})$ in $E$ such that $a_{n}\downarrow 0$ and $d(x_{n},x_{n+p})\leq
a_{n}$ for all $n$ and $p$. Let $b$ be any nonzero positive element in $F$. By
topological uniformly continuity of $f$, there exists some $a=a(b)$ in $E$
such that for all $x,y\in X$ the inequality $d(x,y)<a$ implies $\rho (f(x),f(y))<b$ . Then
there exist elements $b_{n}=b_{n}(\frac{1}{n}b)$ in $E$ such that $%
\rho (f(x_{n}),f(x_{n+p}))<\frac{1}{n}b$ whenever $d(x_{n},x_{n+p})\leq a_{n}\wedge
b_{n}\leq a_{n}$ for all $n$ and $p$. Since $F$ is Archimedean, $\frac{1}{n}%
b\downarrow 0$. Hence, the sequence $f(x_{n})$ is $F$-Cauchy.
\end{proof}

\vskip2mm

\begin{example}
(a) For a vector isometry $f$ between two vector metric spaces $(X,d,E)$ and $(Y,\rho ,F)$, the function $f$ is vectorial uniformly continuous if \ $T_{f}$ is positive and $\sigma$-order continuous.\vskip2mm

(b) For an element $y$ in a vector metric space $(X,d,E)$, the function $f_{y}:X\rightarrow E$ defined by $f_{y}(x)=d(x,y)$ for all $x\in X$ is vectorial uniformly continuous.\vskip2mm

(c) For a subset $A$ of a vector metric space $(X,d,E)$ where $E$ is Dedekind complete, the function $f_{A}:X\rightarrow E$ defined by $f_{A}(x)=d(x,A)=\inf\left\{d(x,y):y\in A\right\}$ for all $x\in X$ is vectorial uniformly continuous.
\end{example}

\vskip2mm

The following theorem enables us to establish a extension property for the functions between vector metric spaces.

\vskip2mm

\begin{theorem}
Let $A$ be $E$-dense subset of a vector metric space $(X,d,E)$ and let $(Y,\rho ,F)$ be a $F$-complete vector metric space where $F$ is Archimedean. If $f:A\rightarrow Y$ is a topological uniformly continuous function, then $f$ has a unique vectorial continuous extension to $X$ which is also this extension is topological uniformly continuous.
\end{theorem}

\begin{proof}
Let $x\in X$. Then there exists a sequence $(x_{n})$ in $A$ such that $x_{n}\overset{d,E}{\longrightarrow }x$.  Since the function $f$ is vectorial uniformly continuous on $A$ by Theorem \ref{Theo5}, the sequence $(f(x_{n}))$ is $F$-Cauchy in $F$-complete vector metric space $Y$. Hence, there exists an element $y\in Y$ such that $f(x_{n})\overset{\rho,F}{\longrightarrow }y$. Define an extension $g$ of $f$ on $X$ by $g(x)=y$. This extension is well-defined, that is, the value of $g$ at $x$ is independent of the particular sequence $(x_{n})$ chosen $E$-convergent to $x$. We need to show that $g$ is topological uniformly continuous on $X$.\vskip2mm
Let $a>0$ in $F$. Choose $b>0$ in $E$ such that for $x,y\in A$ $d(x,y)<b$ implies $\rho (f(x),f(y))<a$. Let $x,y\in X$ satisfy $d(x,y)<b$. Choose two sequences $(x_{n})$ and $(y_{n})$ in $A$ such that $x_{n}\overset{d,E}{\longrightarrow }x$ and $y_{n}\overset{d,E}{\longrightarrow }y$. Then, $d(x_{n},y_{n})\overset{o}{\longrightarrow }d(x,y)$ in $E$. Fix $n_{0}$ such that $d(x_{n},y_{n})<b$ for all $n>n_{0}$. Then $\rho (f(x_{n}),f(y_{n}))<a$ for all $n>n_{0}$. By the vectorial uniformly continuity of $f$, $f(x_{n})$ and $f(y_{n})$ are $F$-Cauchy sequences in $Y$. Since $Y$ is $F$-complete, there exist two points $u$ and $v$ in $Y$ such that $f(x_{n})\overset{\rho,F}{\longrightarrow }u$ and $f(y_{n})\overset{\rho,F}{\longrightarrow }v$. By the definition of the function $g$, we have $g(x)=u$ and $g(y)=v$. Then $\rho(g(x_{n}),g(y_{n}))\overset{o}{\longrightarrow }\rho(g(x),g(y))$ in $F$, and therefore, $\rho (g(x),g(y))\leq b$. This shows that $g$ is topological uniformly continuous function on $X$.
\end{proof}


\section{Vectorial Continuous Function Spaces}

\begin{definition}
Let $X$ be any nonempty set and let $(Y,\rho,F)$ be a vector metric space. Then a sequence $(f_{n})$ of functions from $X$ to $Y$ is said to be uniformly $F$-convergent to a function $f:X\rightarrow Y$, if there exists a sequence $(a_{n})$ in $F$ such that $a_{n}\downarrow 0$ and $\rho(f_{n}(x),f(x))\leq a_{n}$ holds for all $x\in X$ and all $n\in \mathbb{N}$.
\end{definition}

\vskip2mm

Now we give the uniform limit theorem in vector metric spaces.

\vskip2mm

\begin{theorem}\label{7}
Let $(f_{n})$ be a sequence of vectorial continuous functions between two vector metric spaces $(X,d,E)$ and $(Y,\rho,F)$. If $(f_{n})$ is uniformly $F$-convergent to $f$, then the function $f$ is vectorial continuous.
\end{theorem}

\begin{proof}
Let $(x_{n})\subseteq X$ such that $x_{n}\overset{d,E}{\longrightarrow }x$ in $X$. Since $(f_{n})$ is uniformly $F$-convergent to $f$, there is a sequence $(a_{n})$ in $F$ such that $a_{n}\downarrow 0$ and $\rho(f_{n}(x),f(x))\leq a_{n}$ for all $n\in \mathbb{N}$. For each $k\in \mathbb{N}$, there is a sequence $(b_{n})$ in $F$ such that $b_{n}\downarrow 0$ and $\rho(f_{k}(x_{n}),f_{k}(x))\leq b_{n}$ for all $n\in \mathbb{N}$ by the vectorial continuity of $f_{k}$. Note that for $k=n$
$$\rho(f(x_{n}),f(x))\leq\rho(f(x_{n}),f_{n}(x_{n}))+\rho(f(x),f_{n}(x))+\rho(f_{n}(x_{n}),f_{n}(x))\leq 2a_{n}+b_{n}.$$
This implies $f(x_{n})\overset{\rho,F}{\longrightarrow }f(x)$.
\end{proof}

\vskip2mm

Let $A$ be a nonempty subset of a vector metric space $(X,d,E)$. \emph{$E$-diameter} of $A$, $d(A)$, is defined as $\sup\left\{d(x,y):x,y\in A\right\}$. The set $A$ is called \emph{$E$-bounded} if there exists an element $a>0$ in $E$ such that $d(x,y)\leq a$ for all $x,y\in A$. Every $E$-bounded subset of $X$ has an $E$-diameter whenever $E$ is Dedekind complete.

\vskip2mm

\begin{definition}
A function $f:X\rightarrow Y$ between two vector metric spaces $(X,d,E)$ and $(Y,\rho,F)$ is called vectorial bounded if $f$ maps $E$-bounded subsets of $X$ to $F$-bounded subsets of $Y$.
\end{definition}

\begin{theorem}
\label{lem}A function $f:X\rightarrow Y$ between two vector metric spaces $(X,d,E)$ and $(Y,\rho,F)$ is vectorial bounded if there exists a positive operator $T:E\rightarrow F$ such that $\rho(f(x),f(y))\leq T(d(x,y))$ for all $x,y\in X$.
\end{theorem}

\vskip2mm

Let $C_{v}(X,F)$ and $C_{t}(X,F)$ be the collections of all vectorial continuous and topological continuous functions between a vector metric space $(X,d,E)$ and a Riesz space $F$, respectively. By Theorem \ref{1}, $C_{t}(X,F)\subseteq C_{v}(X,F)$ whenever $F$ is Archimedean.

\vskip2mm

\begin{theorem}
The spaces $C_{v}(X,F)$ and $C_{t}(X,F)$ are Riesz spaces with the natural partial ordering defined as $f\leq g$ whenever $f(x)\leq g(x)$ for all $x\in X$.
\end{theorem}

\vskip2mm

Consider an $E$-bounded vector metric space $X$ and a Dedekind
complete Riesz space $F$. Let $C_{v}^{o}(X,F)$ be a subset of
$C_{v}(X,F)$ such that for any $f$ in $C_{v}^{o}(X,F)$, there is a
positive operator $T:E\rightarrow F$ satisfying $\left\vert
f(x)-f(y)\right\vert\leq T(d(x,y))$ for all $x,y\in X$. Since the
Birkhoff inequality (\cite{AB},1.9(2); \cite{Zaanen},12.4(ii))
$$\left\vert f\vee g(x)-f\vee g(y)\right\vert\leq \left\vert
f(x)-f(y)\right\vert +\left\vert g(x)-g(y)\right\vert$$ holds for
all $x,y\in X$, then the subset $C_{v}^{o}(X,F)$ is a Riesz
subspace of $C_{v}(X,F)$. By Theorem \ref{lem}, every $f\in
C_{v}^{o}(X,F)$ is vectorial bounded function. This argument gives
us the following result.

\vskip2mm

\begin{corollary}
The subset $C_{v}^{o}(X,F)$ described above is a vector metric
space equipped with the $F$-valued uniform vector metric defined
as $d_{\infty}(f,g)=\sup_{x\in X}\left\vert f(x)-g(x)\right\vert.$
\end{corollary}


\end{document}